\documentclass{article}

\usepackage[a4paper,left=25mm,top=20mm,right=20mm,bottom=15mm]{geometry}
\usepackage{amsmath, amsfonts, amssymb, amsthm}
\usepackage{color}

\usepackage{mathtools}

\usepackage{caption}
\usepackage{subcaption}

\usepackage{mathrsfs}
\usepackage{graphicx}

\providecommand{\keywords}[1]{\textbf{\textit{Keywords: }} #1}

\newtheorem{theorem}{Theorem}

\newtheorem{defi}[theorem]{Definition}

\newtheorem{lemma}{Lemma}

\newtheorem{corollary}{Corollary}

\begin{document}
\title{Reconstructing unrooted phylogenetic trees from symbolic ternary metrics}
%
%


\author{Stefan Gr\"{u}newald  \\
CAS-MPG Partner Institute for Computational Biology\\Chinese Academy of Sciences Key Laboratory of Computational Biology\\
320 Yue Yang Road, Shanghai 200032,  China\\        Email: \texttt{stefan@picb.ac.cn}
        \and
		  Yangjing Long \\	
		School of Mathematics and Statistics\\
		Central China Normal University, Luoyu Road 152, Wuhan, Hubei
430079, China\\
		  Email: \texttt{yjlong@sjtu.edu.cn}
		  \and     
		  Yaokun Wu \\
Department of Mathematics and MOE-LSC\\
Shanghai Jiao Tong University\\
        Dongchuan Road 800, Shanghai 200240, 
China\\
		  Email: \texttt{ykwu@sjtu.edu.cn}
}

\date{\ }

\maketitle

\abstract{
In 1998, B\"{o}cker and Dress presented a 1-to-1 correspondence 
between symbolically 
dated rooted trees and symbolic ultrametrics.
We consider the corresponding problem for unrooted trees. 
More precisely, given a tree $T$ with leaf set $X$ and a proper 
vertex coloring of its interior vertices, we can map every triple 
of three different leaves to the color of its median vertex. We 
characterize all ternary maps that can be obtained in this way in terms 
of 4- and 5-point conditions, and we show that the corresponding tree 
and its coloring can be reconstructed from a ternary map that 
satisfies those conditions. Further, we give an additional 
condition that characterizes whether the tree is binary, 
and we describe an algorithm that reconstructs general trees in
a  bottom-up fashion. }

\smallskip
\noindent



\sloppy

\section{Introduction}

A phylogenetic tree is a rooted or unrooted tree where the leaves are labeled 
by 
some objects of interest, usually taxonomic units (taxa) like species. The 
edges 
have a positive edge length, thus the tree defines a metric on the taxa set. It 
is a classical result in phylogenetics that the tree can be reconstructed from 
this metric, if it is unrooted or ultrametric. The latter means that the tree 
is 
rooted and all taxa are equally far away from the root. An ultrametric tree is 
realistic whenever the edge lengths are proportional to time and the taxa are species 
that can be observed in the present. In an ultrametric tree, the distance 
between two taxa is twice of the distance between each of the taxa and their last 
common ancestor (lca), hence pairs of taxa with the same lca must have the same 
distance. For three taxa $x,y,z$, it follows that there is no unique maximum 
within their three pairwise distances, thus we have $d(x,y) \le 
\mbox{max}\{d(x,z),d(y,z)\}$. This 3-point condition turns out to be sufficient 
for a metric to be ultrametric, too, and it is the key for reconstructing 
ultrametric trees from their distances. In  1995, Bandelt and 
Steel~\cite{Bandelt1995} observed that 
the complete ordering of the real numbers is not necessary to reconstruct 
trees, 
and they showed that the real-valued distances can be replaced by maps from the 
pairs of taxa into a cancellative abelian monoid. Later, B\" ocker and 
Dress~\cite{Bocker1998} 
pushed this idea to the limit by proving that the image set of the symmetric 
map does not need any structure at all (see Section~\ref{sect:result} for details). 
While this result is useful for understanding how little information it 
takes to reconstruct an ultrametric phylogenetic tree, it was not until recently that 
it turned out to have some practical applications. In 2013, Hellmuth et al. 
\cite{Hellmuth2013} found an alternative characterization of symbolic ultrametrics 
in terms of cographs and showed that, for perfect data, phylogenetic trees can be 
reconstructed from orthology information. By adding some optimization tools, 
this concept was then applied to analyze real data \cite{Hellmuth2015}. 


Motivated by the practical applicability of symbolic ultrametrics, we are 
considering their unrooted version. However, in an unrooted tree there is in 
general no interior vertex associated to a pair of taxa that would correspond 
to 
the last common ancestor in a rooted tree. Instead, there is a median 
associated 
to every set of three taxa that represents, for every possible rooting of the tree, 
a last common ancestor of at least two of the three taxa. 
Therefore, we consider ternary maps from the triples of taxa into an image set 
without any structure. We will show that an unrooted phylogenetic tree with a 
proper vertex coloring can be reconstructed from the function that maps every 
triple of taxa to the color of its median.

In order to apply our results to real data, we need some way to assign a state to
every set of three taxa, with the property that 3-sets with the same median will
usually have the same state. For symbolic ultrametrics, the first real application
was found 15 years after the development of the theory. In addition to the hope that
something similar happens with symbolic ternary metrics, we have some indication that
they can be useful to construct unrooted trees from orthology relations (see Section~\ref{sect:discuss}. 

Consider an unrooted tree $T$ with vertex set $V$, edge set $E$, and leaf 
set 
$X$, and a dating map $t:V\to M^\odot$, where $M^\odot=M\cup \{\odot\}$ such that
$t(x) = \odot$ for all $x\in X$, and $t(v_1)\neq t(v_2)$ if $v_1v_2\in E$.

For any $S=\{x,y\}\in {V \choose 2}$ there is a unique path $[x,y]$ with end 
points $x$ and $y$, and for any 3-set 
$S=\{x,y,z\}\in {V\choose 3}$ there is a unique \textit{triple point} or 
\textit{median} $\text{med}(x,y,z)$ such that 
$[x,y]\cap [y,z]\cap [x,z]=\{\text{med}(x,y,z)\}$. Putting $[x,x]=\{x\}$, 
the 
definition also works, if some or all of $x,y$
and $z$ equal.

Given a phylogenetic tree $T$ on $X$ and a dating map $t: V\to M^\odot $, we can 
define the symmetric symbolic ternary map
$d_{(T;t)}: X\times X\times X\to  M^\odot$
by
$d_{(T;t)}(x,y,z)=t(\text{med}(x,y,z))$.

In this set-up, our question can be phrased as follows: Suppose we are given an arbitrary
symbolic ternary map $\delta: X\times X\times X\to M^\odot$,
can we determine if there is a pair $(T;t)$ for which 
$d_{(T:t)}(x,y,z)=\delta(x,y,z)$ holds for all $x,y,z\in X$?

The rest of this paper is organized as follows.
In Section~\ref{sect:prelim}, we present the basic and relevant concepts used 
in this paper.
In Section~\ref{sect:result} we recall the one-to-one correspondence between 
symbolic ultrametrics and
symbolically dated trees, and introduce our main results 
Theorem~\ref{thm:ultra-ternary} and Theorem~\ref{thm:binary}.

In Section~\ref{sect:proof} we give the proof of Theorem~\ref{thm:ultra-ternary}.
In order to prove our main result, we first introduce the connection between 
phylogenetic trees and quartet systems on $X$ 
in Subsection~\ref{subsect:quartet}. Then we use a graph representation 
to analyze all cases of the map $\delta$ for 5-taxa subsets of $X$ in 
Subsection~\ref{subsect:graphical}.

In Section~\ref{sect:binary} we use a similar method to prove 
Theorem~\ref{thm:binary}, which gives
a sufficient and necessary condition to reconstruct a binary phylogenetic tree on 
$X$.

In Section~\ref{sect:pseudo-cherries}, we give a criterion to identify 
all pseudo-cherries of the underlying tree from a symbolic ternary metric.  
This result makes it possible to reconstruct the tree in 
a bottom-up fashion. 

In the last section we discuss some open questions and future work.

\subsection{Preliminaries}
\label{sect:prelim}

We introduce the relevant basic concepts and 
notation. Unless 
stated otherwise, we will follow the monographs ~\cite{Semple2003} and~\cite{Dress2012}.

In the remainder of this paper, $X$ denotes a finite set of size at least three.

An \textit{(unrooted) tree} $T=(V,E)$ is an undirected connected acyclic graph with 
vertex 
set $V$ and edge set $E$. A vertex of $T$ is a \textit{leaf}
if it is of degree 1, and all vertices with 
degree at least two are \textit{interior} vertices. 

A \textit{rooted tree} $T=(V,E)$ is a tree that contains a distinguished vertex 
$\rho_T \in V$ called the \textit{root}. We define a partial order $\preceq_T$
on $V$ by setting $v\preceq_T w$ for any two vertices $v,w\in V$ for which $v$
is a vertex on the path from $\rho_T$ to $w$. In particular, if $v\preceq_T w$
and $v\neq w$
we call $v$ an \textit{ancestor} of $w$.

An unrooted \textit{phylogenetic tree} $T$ on $X$ is an unrooted tree with 
leaf 
set $X$ that does not contain any vertex of degree 2. It is \textit{binary}, if 
every interior vertex has degree 3.

A rooted phylogenetic tree $T$ on $X$ is a rooted tree with leaf set $X$ that 
does not contain any vertices with in- and out-degree one,
and whose root $\rho_T$ has in-degree zero.
For a set $A\subseteq X$ with cardinality at least 2, we define \textit{the last common 
ancestor} of $A$, denoted by $\text{lca}_T(A)$, to be the unique vertex in $T$
that is the greatest lower bound of $A$ under the partial order $\preceq_T$.
In case $A=\{x,y\}$ we put $\text{lca}_T(x,y) = \text{lca}_T(\{x,y\})$.

Given a set $Q$ of four taxa $\{a,b,c,d\}$,  
there exist always exactly three partitions into two pairs: 
$\{\{a,b\},\{c,d\}\}$,$\{\{a,c\},\{b,d\}\}$ and $\{\{a,d\},\{b,c\}\}$. 
These partitions are called {\em quartets}, and they 
represent the three non-isomorphic unrooted binary trees with leaf set $Q$. These 
trees are usually called quartet trees, and they -- as well as the corresponding 
quartets
--are symbolized by
$ab|cd, ac|bd, ad|bc$ respectively. We use $Q(X)$ to denote the set of all quartets 
with four taxa in $X$. A phylogenetic tree $T$ on $X$ {\em displays} a quartet
$ab|cd \in Q(X)$, if the path from $a$ to $b$ in $T$ is vertex-disjoint with the path from 
$c$ to $d$. The collection of all quartets that are displayed by $T$ is denoted by $Q_T$.

Let $M$ be a non-empty finite set, $\odot$ denotes a special element not 
contained in $M$,
and $M^{\odot}:= M\cup \{\odot\}$. Note that in biology the symbol $\odot$ 
corresponds to
a "non-event" and is introduced for purely technical 
reasons~\cite{Hellmuth2013}. A \textit{symbolic ternary map} is a mapping from $X\times X\times X$ to $M^{\odot}$. Suppose we have a symbolic ternary map $\delta: X\times X\times X 
\to M^{\odot}$,
we say $\delta$ is \textit{symmetric}
if the value of
$\delta(x,y,z)$ is only related to the set $\{x,y,z\}$ but not on the 
ordering of $x,y,z$, i.e., if $\delta(x,y,z)=\delta(y,x,z)=\delta(z,y,x)=\delta(x,z,y)$ for all $x,y,z\in 
X$.
For simplicity, if a map $\delta: X\times X\times X \to M^{\odot}$ is 
symmetric, 
then 
we can define $\delta$ on the set $\{x,y,z\}$ to be $\delta(x,y,z)$.

Let $S$ be a set, we define $|S|$ to be the number of elements in $S$.

\section{Symbolic ultrametrics and our main results}\label{sect:result}
In this section, we first recall the main result  
concerning symbolic ultrametrics by B\"{o}cker and Dress~\cite{Bocker1998}.

Suppose $\delta: X\times X\to M^{\odot}$ is a map. We call 
$\delta$ a \textit{symbolic ultrametric} if it satisfies the following 
conditions:

(U1) $\delta(x,y)=\odot$ if and only if $x=y$;

(U2) $\delta(x,y)=\delta(y,x)$ for all $x,y\in X$, i.e., $\delta$ is symmetric;

(U3) $|\{\delta(x,y),\delta(x,z),\delta(y,z)\}|\leq 2$ for all $x,y,z\in X$; and

(U4) there exists no subset $\{x,y,u,v\} \in {X \choose 4}$ such that

$\delta(x,y) = \delta(y,u) = \delta(u,v) \neq \delta(y,v) = \delta(x,v) = 
\delta(x,u)$.

Now suppose that $T=(V,E)$ is a rooted phylogenetic tree on $X$ and that 
$t:V\to 
M^{\odot}$ is a map such that
$t(x)=\odot$ for all $x\in X$.  
We call such 
a map $t$ a \textit{symbolic dating 
map} for $T$; it is
\textit{discriminating} if $t(u)\neq t(v)$, for all edges $\{u,v\}\in E$. 
Given $(T,t)$, we associate the map $d_{(T;t)}$ on $X\times X$ by setting,
for all $x,y\in X$, $d_{(T;t)}(x,y)=t(\text{lca}_T(x,y))$. Clearly 
$\delta= d_{(T;t)}$ satisfies
Conditions (U1),(U2),(U3),(U4) and we say that $(T;t)$ is a {\em symbolic representation} 
of $\delta$. B\"{o}cker and Dress established in 1998 the following 
fundamental result which gives a
1-to-1 correspondence between symbolic ultrametrics and symbolic 
representations~\cite{Bocker1998},
i.e., the map defined by $(T,t) \mapsto d_{(T,t)}$ is a bijection from the set of
symbolically dated trees into the set of symbolic ternary metrics. 

\begin{theorem}[B\"{o}cker and Dress 1998~\cite{Bocker1998}]\label{thm:rooted}
 Suppose $\delta: X\times X\to M^{\odot}$ is a map. Then there is a 
discriminating symbolic representation of
 $\delta$ if and only if $\delta$ is a symbolic ultrametric. Furthermore, up to 
isomorphism, this representation is unique.
\end{theorem}

Similarly, we consider unrooted trees. Suppose that $T=(V,E)$ is an unrooted 
tree 
on $X$ and that $t:V\to M^{\odot}$ is a symbolic dating map, i.e.,
$t(x)=\odot$ for all $x\in X$, it is discriminating if $t(x)\neq t(y)$ for all $(x,y)\in E$.
Given the pair $(T;t)$, we associate the map 
$\delta_{(T;t)}$ on $X\times X\times X$ by setting,
for all $x,y,z\in X$, $\delta_{(T;t)}(x,y,z)=t(\text{med}(x,y,z))$.

Before stating our main results, we need the following definition:
\begin{defi}[$n$-$m$ partitioned]
 Suppose $\delta: X\times X\times X\to M^{\odot}$ is a symmetric map. We say 
that a subset $S$ of $X$ is \textit{$n$-$m$ partitioned (by $\delta$)}, if
 among all the 3-element subsets of $S$, there are in 
total 2 different values of $\delta$,  
and $n$ of those 3-sets are mapped to one value while all other $m$ 3-sets 
are mapped to the other value.
\end{defi}
Note that $S$ can be $n$-$m$ partitioned, only when $\binom{|S|} { 3} = m+n$. 
\begin{defi}[symbolic ternary metrics]
We say $\delta:X\times X\times X\to M^{\odot}$ is a \textit{symbolic 
ternary metric}, if the following conditions hold.

(1) $\delta$ is symmetric, i.e., 
$\delta(x,y,z)=\delta(y,x,z)=\delta(z,y,x)=\delta(x,z,y)$ for all $x,y,z\in X$.

(2) $\delta(x,y,z)=\odot$ if and only if $x=z$ or $y=z$ or $x=y$.

(3) for any distinct $x,y,z,u$ we have 
$$|\{\delta(x,y,z),\delta(x,y,u),\delta(x,z,u),\delta(y,z,u)\}|\leq 2,$$
and when the equality holds then $\{x,y,z,u\}$ is 2-2 partitioned by $\delta$.

(4) there is no distinct 5-element subset $\{x,y,z,u,e\}$ of $X$ which
is 5-5 partitioned by $\delta$.
 
\end{defi}

We will refer to these conditions throughout the paper. 

Our main result is:

\begin{theorem} \label{thm:ultra-ternary}

There is a 1-to-1 correspondence between the discriminating symbolically 
dated phylogenetic trees and 
the symbolic ternary metrics on $X$.
\end{theorem}

Let $\delta$ be a ternary symbolic ultrametric on $X$. Then we call 
$\delta$ {\em fully resolved}, if the following condition holds: 

(*) If $|\{\delta(x,y,z),\delta(x,y,u),\delta(x,z,u),\delta(y,z,u)\}|= 1$, then 
there exists $e\in X$ such that $e$ can resolve $xyzu$.
i.e., the set $\{x,y,z,u,e\}$ is 4-6 partitioned by $\delta$.

Now we can characterize ternary symbolic ultrametrics  that 
correspond to binary phylogenetic trees:  

\begin{theorem} \label{thm:binary}

There is a 1-to-1 correspondence between the discriminating symbolically 
dated binary phylogenetic trees and the 
fully resolved symbolic ternary metrics on $X$.

\end{theorem}

%
%
%
%
%
%
%
%

\section{Reconstructing a symbolically dated phylogenetic tree}
\label{sect:proof}

The aim of this section is to prove Theorem~\ref{thm:ultra-ternary}.

\subsection{Quartet systems}
\label{subsect:quartet}

We will use quartet systems to prove Theorem~\ref{thm:ultra-ternary}. 
In 1981, Colonius and Schulze~\cite{Colonius1981} found that, 
for a quartet system $Q$ on 
a finite taxa set X, there is a phylogenetic tree $T$ on $X$ such that 
$Q=Q_T$, if and only if certain conditions on subsets of $X$ with up to 
five elements hold. The following theorem (Theorem 3.7 in~\cite{Dress2012}) 
states their result.

%



A quartet system $Q$ is \textit{thin}, if for every 4-subset ${a, b, c, 
d}\subseteq X$,
at most one of the three quartets $ab|cd$, $ac|bd$ and $ad|bc$ is contained
in $Q$. It is \textit{transitive}, if for any 5 distinct elements $a, b , c , d , e \in X$, the 
quartet
$ab |cd$ is in $Q$ whenever both of the quartets $ab |ce$ and $a b |d e$ are 
contained in $Q$. It is \textit{saturated}, if for any
five distinct elements $a$ , $b$ , $c$ , $d$ , $e \in X$ with $a b 
|c d \in Q$, at least one of
the two quartets $a e|c d$ and $a b |c e$ is also in $Q$.

\begin{theorem}\label{thm:quartet2}
 A quartet system $Q \subseteq Q(X)$ is of the form $Q = Q(T)$ for some
phylogenetic tree $T$ on $X$ if and only if 
$Q$ is thin, transitive and saturated.

\end{theorem}

We can encode a phylogenetic tree on $X$ in terms of a quartet system by taking 
all the quartets displayed by the tree, as 
two phylogenetic 
trees on $X$ are isomorphic if and only if
the associated quartet systems coincide~\cite{Dress2012}.
Hence, a quartet system that satisfies Theorem~\ref{thm:quartet2} uniquely determines a phylogenetic
tree. 
%

%

%

\subsection{Graph representations of a ternary map}
\label{subsect:graphical}

Suppose we have a symmetric map $\delta: X\times X\times X \to M^{\odot}$.
Then we can represent the restriction of $\delta$ to all 3-element subsets 
of any 5-element subset $\{x,y,z,u,v\}$ of $X$ by an edge-colored complete 
graph on the 5 vertices 
$x,y,z,u,v$. For any distinct $a,b,c,d\in \{x,y,z,u,v\}$, edge $ab$ and edge 
$cd$ have the same color if and only if the value of $\delta$ for  
$\{x,y,z,u,v\}\setminus \{a,b\}$ is the same as  for 
$\{x,y,z,u,v\}\setminus \{c,d\}$.
It follows from Condition (3) in the definition of a symbolic ternary metric that, 
for any vertex of the graph, 
either 2 incident edges have one color and the other 2 edges have another 
color,
or all 4 incident edges have the same color. By symmetry, there are exactly five 
non-isomorphic graph representations.

\begin{lemma}\label{lem:graph}
Let the edges of a $K_5$ be colored such that for each vertex, the 4 
incident edges  
are either colored by the same color, or 2 of them colored by one color and 
the other 2 by another color.
Then there are exactly 5 non-isomorphic colorings, and they are depicted in 
Figure~\ref{fig:5-subset-full}. 

\begin{figure}[ht!]
\centering  
\includegraphics[width=0.9\textwidth]{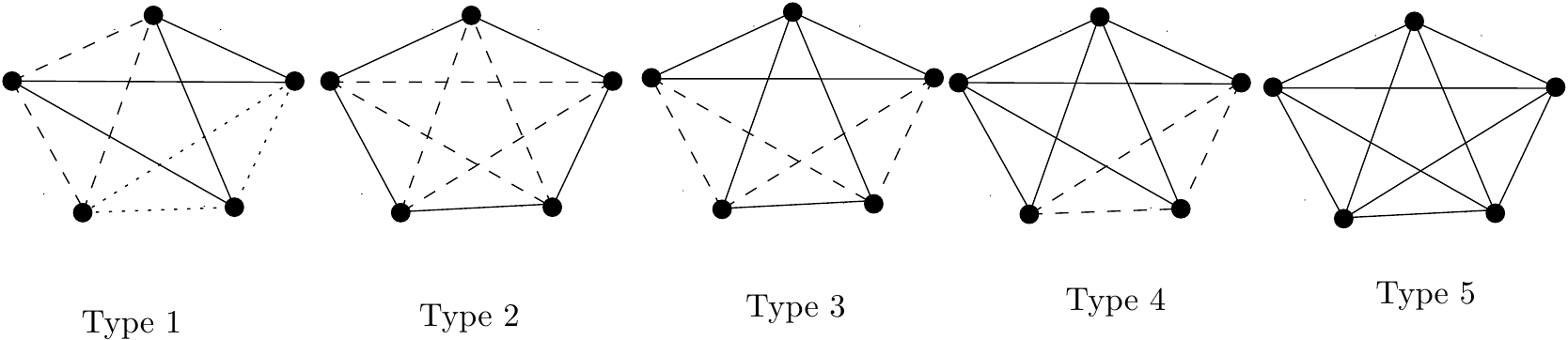} 
\caption{The 5 non-isomorphic colorings of $K_5$ for which every color class induces an Eulerian graph. Note that in type 1 there 
are 3 types of edges, solid edges, dotted edge
and dashed edges} 
 \label{fig:5-subset-full} 
\end{figure}

\end{lemma}

\begin{proof}

It follows from the condition on the coloring that every color class induces 
an {\em Eulerian} subgraph (a graph where all vertices have even 
degree) of $K_5$. Therefore, ignoring isolated vertices, every such induced subgraph 
either is a cycle or it contains a vertex of degree four. Since there are only 
ten edges, the only way to have three color classes is two triangles and one 
4-cycle. In that case each of the triangles must contain two non-adjacent 
vertices of the 4-cycle and the vertex that is not in the 4-cycle, thus we get a 
coloring isomorphic to Type 1 in Figure~\ref{fig:5-subset-full}. If there are 
exactly two color classes, then one of them has to be a cycle and the other one 
its complement. This yields Types 2, 3, and 4, if the length of the cycle is 5, 
4, and 3, respectively. Finally, if there is only one color, we get Type 5.

\end{proof}

Note that the vertices are not labeled and it does not matter which color we 
are using. 

We will prove Theorem~\ref{thm:ultra-ternary} by obtaining a quartet system 
from any symbolic ternary metric. More precisely, we say that the symbolic 
ternary metric $\delta$ on $X$ {\em generates} the quartet $xy|zu$ if either 
$\delta(x,z,u)=\delta(y,z,u) \neq \delta(x,y,z)=\delta(x,y,u)$, or 
$|\{\delta(x,y,z),\delta(x,y,u),\delta(x,z,u),\delta(y,z,u)\}|=1$ and there is $e \in X$ 
such that 
$$\delta(x,y,e)=\delta(x,y,z)=\delta(x,y,u)=\delta(x,z,u)=\delta(z,u,e)=\delta(y
,z,u)$$ 
 $$\neq \delta(x,u,e)=\delta(x,z,e)=\delta(y,z,e)=\delta(y,u,e).$$
In the latter case, we say that $e$ {\em resolves} $x,y,z,u$. Note that the 3-sets 
obtained by adding $e$ to the pairs of the generated quartet both have the same 
$\delta$-value as the subsets of $\{x,y,z,u\}$. 
The following lemma 
will show that the set of all quartets generated by a symbolic ternary metric is 
thin. 

\begin{lemma} \label{lem:quartet2}
Let $\delta: X\times X\times X \to M^{\odot}$ be a symbolic ternary metric 
and let $x,y,z,u\in X$ be four different taxa with 
$|\{\delta(x,y,z),\delta(x,y,u),\delta(x,z,u),\delta(y,z,u)\}|= 1$. Let 
$e,e' \in X-\{x,y,z,u\}$ such that $\{x,y,z,u,e\}$ and $\{x,y,z,u,e'\}$ are 
both 4-6-partitioned, and let 
$$\delta(x,y,e)=\delta(x,y,z)=\delta(x,y,u)=\delta(x,z,u)=\delta(z,u,e)=\delta(y
,z,u)$$ 
 $$\neq \delta(x,u,e)=\delta(x,z,e)=\delta(y,z,e)=\delta(y,u,e).$$
Then we also have 
$$\delta(x,y,e')=\delta(x,y,z)=\delta(x,y,u)=\delta(x,z,u)=\delta(z,u,e')=\delta(y
,z,u)$$ 
 $$\neq \delta(x,u,e')=\delta(x,z,e')=\delta(y,z,e')=\delta(y,u,e').$$
\end{lemma}

\begin{proof}

We already know that
$|\{\delta(x,y,z),\delta(x,y,u),\delta(x,z,u),\delta(y,z,u)\}|= 1$ and 
 $\{x,y,z,u,e'\}$ is 4-6-partitioned.
 
So there are three possible cases for the values of $\delta$ on  $\{x,y,z,u,e'\}$.

(1) $\delta(x,y,z)=\delta(x,y,u)=\delta(x,z,u)=\delta(y,z,u)$ and the rest 6 are equal. 
Then consider $\delta$ on $\{x,y,z,e'\}$, if it is 1-3 partitioned instead of 2-2 partitioned, 
then it contradicts the definition of symbolic ternary metric, thus this case would not happen.

(2) $\delta(x,y,z)=\delta(x,y,u)=\delta(x,z,u)=\delta(y,z,u) =\delta(e',a,b)=\delta(e',a,c)$ where
$\{a,b,c\}\in {\{x,y,z,u\} \choose 3}$. There are totally 12 different cases. Since $x,y,z,u$ are symmetric, w.l.o.g., we assume
$\delta(x,y,z)=\delta(x,y,u)=\delta(x,z,u)=\delta(y,z,u)=\delta(e',x,y)=\delta(e',x,z)$ and the rest are equal.
Then consider $\delta$ on $\{e',x,y,z\}$, if it is 1-3 partitioned instead of 2-2 partitioned, 
then it contradicts to the definition of symbolic ternary metric, thus this case would not happen.

(3) $\delta(x,y,z)=\delta(x,y,u)=\delta(x,z,u)=\delta(y,z,u) =\delta(e',a,b)=\delta(e',c,d)$ where
$\{a,b,c,d\}\in {\{x,y,z,u\} \choose 4}$. There are totally 3 different cases.
Suppose the statement of the lemma is wrong. Then because $x,y,z,u$ are symmetric w. l. 
o. g. we can assume that  
$$\delta(x,z,e')=\delta(x,y,z)=\delta(x,z,u)=\delta(y,u,e')=\delta(x,y,u)=\delta(y
,z,u)$$ 
 $$\neq \delta(x,u,e')=\delta(x,y,e')=\delta(y,z,e')=\delta(z,u,e').$$

 
 Case (3a): $\delta(x,u,e)\neq \delta(x,u,e')$. We assume that 
 $\delta(x,u,e)$ is dashed, $\delta(x,u,e')$ is dotted, and 
 $\delta(x,y,z)$ is solid. 
Since $\delta $ is a symbolic 
ternary metric, by Lemma~\ref{lem:graph}, the graph representation of 
 $\{y,z,u,e,e'\}$ has to be Type 1, so the color classes are 
 one 4-cycle and two 3-cycles. The values of $\delta$ 
 for the sets that contain at most one of $e$ and $e'$ are shown in 
 Figure~\ref{fig:example_lem2}. There  is a path of 
 length 3 that is colored with $\delta(x,y,z)$ (solid), and there are paths of length 2
 colored with $\delta(x,u,e)$ (dashed) and $\delta(x,u,e')$ (dotted), 
 respectively. It follows that 
 we only can get Type 1 by coloring the edges connecting the end vertices 
 of each of those paths with the same color as the edges on the path. We get 
  $\delta(u,e,e')=\delta(x,y,z)$ (solid),  $\delta(z,e,e')=\delta(x,u,e')$ (dotted), and  
  $\delta(y,e,e')=\delta(x,u,e)$ (dashed). Now doing the same analysis for $\{x,y,z,e,e'\}$ 
  yields $\delta(z,e,e')=\delta(x,u,e)$, in contradiction to $\delta(z,e,e')=\delta(x,u,e')$. 

%
\begin{figure}[ht!]
\centering  
\includegraphics[width=0.4\textwidth]{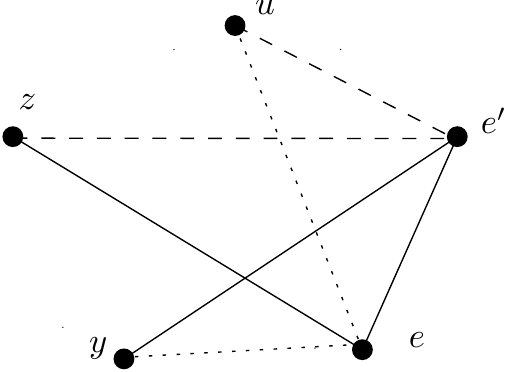} 
\caption{The partial coloring of $K_5$ as described in Case (a).} 
\label{fig:example_lem2} 
\end{figure}

Case (3b): $\delta(x,u,e)= \delta(x,u,e')$. The graph representation of $\{y,z,u,e,e'\}$ 
can be obtained from Figure~\ref{fig:example_lem2} by identifying the colors dashed and dotted. 
It contains a path of length 3 
of edges colored with $\delta(x,y,z)$ and a path of length 4 of edges colored 
with $\delta(x,u,e)$. Since a path is not an Eulerian graph, both colors must be 
used for at least one of the remaining three edges, thus Type 1 is not possible.

Due to $\delta$ being a symbolic ternary map,
$\{y,z,u,e,e'\}$ is not 5-5-partitioned, and since there are only two colors
with at least 4 edges of one color and 5 edges of the other color, Lemma~\ref{lem:graph}
implies that the corresponding graph representation must be Type 2 and
therefore, $\{y,z,u,e,e'\}$ is 4-6-partitioned.
We get $\delta(u,e,e')=\delta(x,y,z)$, 
$\delta(y,e,e')=\delta(z,e,e')=\delta(x,u,e)$. Now we consider the graph representation 
of $\{x,z,u,e,e'\}$, and we observe that the edges colored with $\delta(x,y,z)$ contain a 
path of length 4, and the edges colored with $\delta(x,u,e)$ contain a 5-cycle. Hence, 
$\{x,z,u,e,e'\}$ is 5-5-partitioned, in contradiction to Condition (4). 
\end{proof}

\begin{proof}[Proof of Theorem~\ref{thm:ultra-ternary}]
By the definitions of the median, the ternary map $\delta_{(T;t)}$ 
associated with a discriminating
symbolically dated phylogenetic tree $(T,t)$ on $X$ satisfies Conditions 
(1) and (2). For any distinct leaves $x,y,z,u$, the smallest subtree of $T$ connecting 
those four leaves has at most two vertices of degree larger than two. If there are two such vertices,
then each of them has degree 3 and is the median of two 3-sets in $\{x,y,z,u\}$.
Therefore, $\delta_{(T;t)}$ satisfies Condition (3).

For any 5 distinct leaves, the smallest subtree of $T$ connecting them either has 
three vertices of degree 3, or one vertex of degree 3 and one of degree 4, or one 
vertex of degree 5, while all other vertices have degree 1 or 2. The first case is depicted in 
Figure~\ref{fig:5-point}. There $v_1$ is the median for the 3-sets that contain $x_1$ and $x_2$, 
$v_3$ is the  median for the 3-sets that contain $z_1$ and $z_2$, and $v_2$ is the median for the 
remaining four 3-sets. Hence, either $\{x_1,x_2,y,z_1,z_2\}$ is 4-6-partitioned 
(if $t(v_1)=t(v_3)$), or 
there are three different values of $\delta_{(T;t)}$ within those five taxa. For the other two cases, 
the set of five taxa is either 3-7-partitioned or $\delta_{(T;t)}$ is constant on all its subsets with 3 taxa. 
Hence, no subset of $X$ of cardinality 
five is 5-5-partitioned by $\delta_{(T;t)}$, thus $\delta_{(T;t)}$ satisfies Condition (4). 

%
%
On the other hand, let $\delta$ be a symbolic ternary metric on $X$. 
By Lemma~\ref{lem:graph}, taking any 5-element subset of $X$,
 the possible graph representations of the delta system satisfying (1), (2), and (3) 
are 
shown in Figure~\ref{fig:5-subset-full}. Except for Type 2, all other types 
satisfy (4).

For the first type, the delta system is 
$\delta(y,z,u)=\delta(x,y,z)=\delta(w,y,z)\neq 
\delta(w,x,z)\\
=\delta(w,x,u)=\delta(w,x,y)\neq \delta(x,z,u)=\delta(w,z,u)
=\delta(x,y,u)=\delta(w,y,u)\neq \delta(y,z,u)$. 
The corresponding quartet system is $\{xw|yu,xw|zu,xw|yz,uw|yz,xu|yz\}$.

For the third type, the delta system is 
$\delta(y,z,u)=\delta(x,z,u)=\delta(w,z,u)= 
\delta(x,y,w)\\
=\delta(y,u,w)=\delta(y,z,w)\neq \delta(x,y,z)=\delta(w,x,z)
=\delta(x,y,u)=\delta(w,x,u)$. 
The corresponding quartet system is $\{wy|xu,wy|xz,xy|zu,wy|zu,wx|zu\}$.

For the fourth type, the delta system is 
$\delta(w,x,z)=\delta(w,x,u)=\delta(w,x,y)\\
\neq 
\delta(y,z,u)=\delta(x,y,z)=\delta(x,z,u)= \delta(w,z,u)=\delta(w,y,z)
=\delta(x,y,u)=\delta(w,y,u)$. 
The corresponding quartet system is $\{xw|yu,xw|zu,xw|yz\}$.

For the fifth type, the delta system is 
$\delta(y,z,u)=\delta(x,y,z)=\delta(w,y,z)= 
\delta(w,x,z)\\
=\delta(w,x,u)=\delta(w,x,y)= \delta(x,z,u)=\delta(w,z,u)
=\delta(x,y,u)=\delta(w,y,u)$. 
The corresponding quartet system is $\emptyset$.

All quartet systems are thin, transitive, and saturated. Indeed, the delta systems 
of Types 1 and 3 generate all quartets displayed by a binary tree, Type 4 generates 
all quartets  displayed by a tree with exactly one interior edge, and Type 5 corresponds to the  
star tree with 5 leaves. Now we take the union of all quartets generated 
by $\delta$. The resulting quartet system is thin in view of Lemma~\ref{lem:quartet2},
it is easy to see that it is also transitive and saturated, and by Theorem~\ref{thm:quartet2}, 
every delta system satisfying Conditions (1), (2), (3) and (4) 
uniquely determines a phylogenetic tree $T$ on $X$. 

It only remains to show that two 3-element subsets of $X$ that have the same median in $T$ must be 
mapped to the same value of $\delta$, since then we can define $t$ to be the dating map with 
$t(v)=\delta(x,y,z)$ for every interior vertex $v$ of $T$ and for all 3-sets $\{x,y,z\}$ whose median is $v$. 
It suffices to consider two sets which intersect in two taxa, as the general 
case follows by exchanging one taxon up to three times. We assume that $v$ is the median of both, 
$\{w,x,y\}$ and $\{w,x,z\}$. If $\delta(w,x,y) \neq \delta(w,x,z)$, then by the definition of
symbolic ternary metric, $\{x,y,w,z\}$ is 2-2-partitioned by $\delta$, thus $\delta$ generates one of the quartets 
$wy|xz$ and $wz|xy$, thus $T$ must display that quartet. However, in both cases $\{w,x,y\}$ and $\{w,x,z\}$ 
do not have the same median in $T$. We also claim that the associated $t$ is discriminating. Suppose otherwise,
there is an edge $uv$ in $T$ such that $t(u)=t(v)$. Thus for any four leaves $x_1,x_2,y_1,y_2$ such that
$\text{med}(x_1,x_2,y_1) = \text{med}(x_1,x_2,y_2) = u$
and $\text{med}(y_1,y_2,x_1) = \text{med}(y_1,y_2,x_2) = v$, 
the quartet $x_1x_2|y_1y_2$ is displayed by $T$.
Further we have
$\delta(x_1,x_2,y_1)=\delta(x_1,x_2,y_2)=\delta(y_1,y_2,x_1)=\delta(y_1,y_2,x_2)$.
Because the quartet $x_1x_2|y_1y_2$ is generated by $\delta$, there is a leave $e$ which resolves $x_1x_2y_1y_2$.
Since $uv$ is an edge in $T$, there exists $i$ and $j$
such that $\text{med}(x_i,y_j,e) \in \{u,v\}$, we have $\delta(x_i,y_j,e)=\delta(x_1,x_2,y_1)$,
which means $e$ cannot resolve $x_1x_2y_1y_2$, a contradiction.
Hence a thin, transitive and saturated quartet system uniquely determines a discriminating symbolic dated tree.
%
\end{proof}

\begin{figure}[ht!]
\centering  
\includegraphics[width=0.5\textwidth]{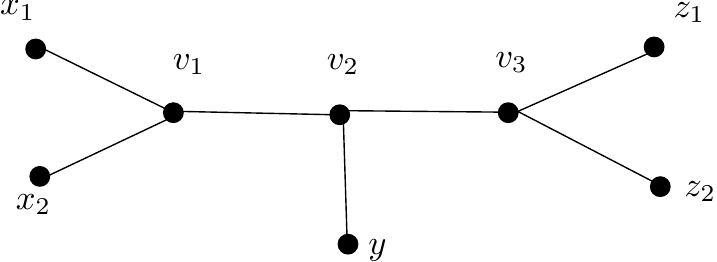} 
\caption{The leaves and median vertices for a 5-taxa binary tree.} 
 \label{fig:5-point} 
\end{figure}

The set of quartets generated by Type 2 does not satisfy the condition of being saturated 
from Theorem~\ref{thm:quartet2}.
Without loss of generality, label the vertices by $x,y,z,u,w$ as in Figure~\ref{fig:5-subset-2}. 
Then the quartet system is $\{yw|zu,xu|yz,xz|uw,xy|zw,xw|yu\}$. In order 
to be saturated, the presence of $xu|yz$ would induce that we have $xw|yz$ 
or $xu|yw$, but we have $xy|zw$, $xw|yu$ instead.

\begin{figure}[ht!]
\centering  
\includegraphics[width=0.4\textwidth]{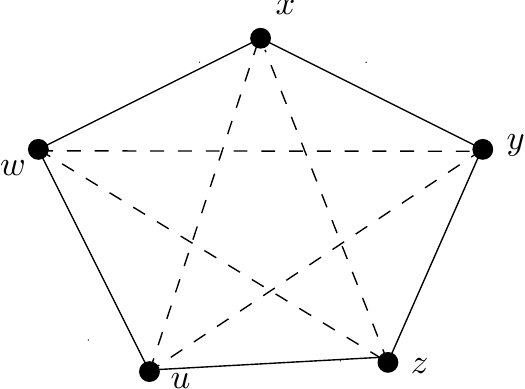} 
\caption{The graph representation of a ternary map satisfying Conditions (1), (2), and (3), but not (4).} 
 \label{fig:5-subset-2} 
\end{figure}

\section{Reconstructing a binary phylogenetic tree}
\label{sect:binary}

The aim of this section is to prove Theorem~\ref{thm:binary}.

A quartet system $Q$ on $X$ is \textit{complete}, if
$$|\{Q\cap \{ab|cd,ac|bd,ad|bc\}\}| = 1$$
 holds for all  $\{a,b,c,d\}\in {X \choose 4}$. Using the easy observation that 
 a phylogenetic tree is binary if and only if it displays a quartet for every 4-set, 
 the following result is a direct consequence of Theorem~\ref{thm:quartet2}.

\begin{corollary}~\label{cor:quartet}
  A quartet system $Q \subseteq Q(X)$ is of the form $Q = Q(T)$ for some
binary phylogenetic tree $T$ on $X$ if and only if 
$Q$ is complete, transitive, and saturated. 
\end{corollary}


%
%
%

Condition (*) ensures that a ternary metric 
$\delta$ 
generates a quartet for every set of four taxa, even if $\delta$ is constant on all 
of its 3-taxa subsets. In view of Lemma~\ref{lem:quartet2} we also have that 
$\delta$ can not generate two different quartets for the same 4-set. Hence, 
we have the following corollary. 

\begin{corollary}\label{cor:complete}
 A symbolic ternary metric $\delta:X\times X \times X\to M^{\odot}$ 
that satisfies Condition~(*) generates a complete quartet system on $X$.
\end{corollary}

Now we prove Theorem~\ref{thm:binary}.

\begin{proof}

Let $(T,t)$ be a symbolically dated binary phylogenetic tree. 
By Theorem~\ref{thm:ultra-ternary}, $\delta_{(T;t)}$ is a
symbolic ternary metric. Since $T$ is binary, it displays a 
quartet for every 4-taxa subset $\{x,y,z,u\}$ of $X$. Assume 
that $T$ displays $xu|yz$, thus 
$\text{med}(x,y,z) \neq \text{med}(x,y,u)$. If 
$|\{\delta_{(T;t)}(x,y,z),\delta_{(T;t)}(x,y,u),\delta_{(T;t)}(x,z,u),\delta_{(T;t)}(y,z,u)\}|= 1$, 
then \\there is at least one vertex $v$ on the path in $T$ connecting 
$\text{med}(x,y,z)$ and $\text{med}(x,y,u)$ with $t(v) \neq t(\text{med}(x,y,z))$, as 
$t$ is discriminating. Hence, there is a leaf $e \in X$ such that
$v = \text{med}(x,y,e)$. 
It follows that the set $\{x,y,z,u,e\}$ is 4-6 partitioned by $\delta_{(T;t)}$, 
thus $\delta_{(T;t)}$ satisfies Condition (*).

On the other hand, if $\delta$ is a symbolic 
ternary metric on $X$  and satisfies (*),
then by Corollary~\ref{cor:complete}, it corresponds to a unique complete 
quartet system, thus it encodes a binary phylogenetic tree $T$ in view of 
Corollary~\ref{cor:quartet}. 
As in the last paragraph of the proof of 
Theorem~\ref{thm:ultra-ternary}, we can define a 
dating map $t$ by $t(v)=\delta(x,y,z)$ for every interior vertex $v$ of $T$ and for all 
3-sets $\{x,y,z\}$ whose median is $v$. Hence, $(T,t)$ is a symbolically dated binary 
phylogenetic tree.  
\end{proof}

\section{The recognition of pseudo-cherries}
\label{sect:pseudo-cherries}
In Theorem~\ref{thm:ultra-ternary}, we have established a 1-to-1 correspondence 
between symbolically dated phylogenetic trees and symbolic ternary metrics on $X$, and 
a bijection is given by mapping $(T,t)$ to $d_{(T,t)}$. To get the inverse of this map, we can 
first compute the set of all quartets generated by a symbolic ternary metric, and then 
apply an algorithm that reconstructs a phylogenetic tree from the collection of all its displayed 
quartets. Finally, the dating map is defined as in our proof of Theorem~\ref{thm:ultra-ternary}. 
This approach would correspond to first extracting rooted triples from a symbolic ultrametric 
and then reconstruct the rooted tree (see Section 7.6 of~\cite{Semple2003}). However, a 
more direct way to reconstruct the corresponding tree from a symbolic ultrametric was presented 
in~\cite{Hellmuth2013}. It is based on identifying maximal sets of at least two taxa that are 
adjacent to the same interior vertex, so-called {\em pseudo-cherries}. These can iteratively 
be identified into a single new taxon, thereby reconstructing the corresponding tree in a bottom-up fashion.

The main advantage of such an algorithm is that it might be used to heuristically  construct a tree, even 
if the input is not a symbolic ternary metric. In terms of running time, using pseudo-cherries is also 
slightly better than having to compute all $O(n^4)$ quartets, but the $O(n^3)$ input size limits the speed 
of every algorithm that deals with ternary maps. 

%

Here we only show how to find the pseudo-cherries of $T$ from a symbolic ternary metric 
$d_{(T,t)}$. An algorithm to reconstruct $T$ can be designed exactly as in~\cite{Hellmuth2013} and is 
therefore omitted. 
We point out that it is not necessary to check Condition (4) of a symbolic ternary metric, as a violation would 
make the algorithm recognize that the ternary map does not correspond to a tree.

Given an arbitrary symbolic ternary map  $\delta: X\times X\times X\to M^\odot$ satisfying 
Conditions (1), (2), and (3).
For $x,y \in X$ and $m\in  M^{\odot}$, we say that $x$ and $y$ are {\em $m$-equivalent}, 
if there is $z \in X$ such that $\delta(x,y,z)=m$, and for $u,v \in X-x-y$, $\delta (x,u,v)=m$ if and only if 
$\delta (y,u,v)=m$. 
\begin{lemma} \label{lem:transitive}
If $x$ and $y$ are $m$-equivalent and $y$ and $z$ are $m'$-equivalent, then $m=m'$ and $x$ and $z$ are 
$m$-equivalent. 
\end{lemma} 
\begin{proof}
Assume $\delta(x,y,z) \neq m$. Then let $u \in X$ with $\delta(x,y,u)=m$. Since $\delta$ is not constant on 
$\{x,y,z,u\}$, this 4-set must be 2-2-partitioned, thus exactly one of $\delta(x,u,z)=m$ and $\delta(y,u,z)=m$ 
must hold, in contradiction to $x$ and $y$ being $m$-equivalent. Hence, we have $\delta(x,y,z) = m$, and 
by symmetry we also have $\delta(x,y,z) = m'$, thus $m=m'$. In order to verify that  $x$ and $z$ must be 
$m$-equivalent, we have already shown $\delta(x,y,z) = m$. For $w,w' \in X-\{x,y,z\}$, we have
$\delta(x,w,w')=m$ if and only if $\delta(y,w,w')=m$, since $x$ and $y$ are $m$-equivalent, and we have 
$\delta(y,w,w')=m$ if and only if $\delta(z,w,w')=m$, since $y$ and $z$ are $m$-equivalent. 
Finally, we have $\delta(x,y,w)=m$ if and only if $\delta(x,z,w)=m$ if and only if $\delta(y,z,w)=m$. 
Hence $x$ and $z$ are $m$-equivalent. 
\end{proof}

We say $x,y\in X$ are {\em $\delta$-equivalent}, denoted by {\em $x\sim_{\delta} y$}, if there exists 
$m\in M^\odot$
such that $x$ and $y$ are $m$-equivalent.

\begin{lemma}
 The relation of being $\delta$-equivalent is an equivalence relation.
\end{lemma}

\begin{proof}
 For any $x\in X$, since $\delta(x,x,y)=\odot$ for any $y\in X$, by definition $x$ and $x$ are $\odot$-equivalent,
 hence $x$ and $x$ are $\delta$-equivalent. Hence $\sim_{\delta}$ is reflexive. For any $x\sim_{\delta} y$, we know that there exists an $m\in M^{\odot}$
 such that $x$ and $y$ are $m$-equivalent. Since $\delta$ is symmetric, by the definition of $m$-equivalent, $y$ and $x$
 are also $m$-equivalent, thus $y\sim_{\delta} x$. Hence, $\sim_{\delta}$ is symmetric.
 To prove the transitivity of $\sim_{\delta}$, assume $x\sim_{\delta} y$ and $y\sim_{\delta} z$, by Lemma~\ref{lem:transitive}
 we know that $x\sim_{\delta} z$. Therefore, $\delta$-equivalent is an equivalence relation.
\end{proof}

Suppose that $T$ is a phylogenetic tree on $X$. Let $C\subseteq X$ be a
subset of $X$ with $|C| \ge 2$. 
We call $C$ a {\em pseudo-cherry} of $T$, if there is an interior vertex 
$v$ of $T$ such that $C$ is the set of all leaves adjacent to $v$ . 

\begin{theorem}
If $(T;t)$ is a symbolically dated phylogenetic tree, then 
a non-empty subset $C$ of $X$ is a non-trivial equivalence class of $\sim_{\delta_{(T,t)}}$
if and only if $C$ is a pseudo-cherry of $T$.

\end{theorem}

\begin{proof}
 For the ease of notation, we let $\delta = \delta_{(T,t)}$. 
 
 Since $t$ is discriminating, the definition of a pseudo-cherry 
 immediately implies that any pseudo-cherry of $T$ must be contained in a 
 non-trivial equivalence class of $\sim_{\delta}$.

Conversely, if a non-trivial equivalence class $C$ of $\sim_{\delta}$ is not a pseudo-cherry,
then there are $x_1,x_2 \in C$ such that the path in $T$ that contains $x_1$ and $x_2$ 
has length at least 3,  
and since $t$ is discriminating, it has at least 2 interior vertices labeled by two 
different elements of $M$. Suppose that all elements of $C$ are $m$-equivalent, and 
that $v$ is an interior vertex on the path from $x_1$ to $x_2$ such that $t(v)=m'$ and 
$m' \neq m$. Further, let $y \in X$ such that $v=\text{med}(x_1,x_2,y)$. Since $x_1$ and 
$x_2$ are $m$-equivalent, there is $z \in X$ such that $\delta(x_1,x_2,z)=m$. Then the 
median $u$ of $x_1,x_2,z$ is also on the path from $x_1$ and $x_2$, and we assume 
without loss of generality that 
$u$ is on the path from $x_1$ to $v$. It follows that $\delta(x_1,y,z) =m$ but $\delta(x_2,y,z)=m'$
in contradiction to $x_1\sim_mx_2$.
\end{proof}

\section{Discussions and open questions}
\label{sect:discuss}

The proofs of our main results heavily rely on extracting the corresponding quartet set from 
a symbolic ternary metric and then checking that our Conditions (3) and (4) guarantee the 
quartet system to be thin, transitive, and saturated, and adding (*) makes the quartet system complete. 
The conditions look like (3) corresponds to thin and transitive, (4) to saturated, and (*) to complete. 
However, this is not true, and removing (4) from Theorem~\ref{thm:binary} does not necessarily 
yield a transitive complete quartet system. While for five taxa, a 5-5-partition yields the only 
non-saturated complete transitive quartet system, Lemma~\ref{lem:quartet2} does not hold without 
Condition (4). Indeed the ternary map that is visualized in Figure~\ref{fig:new} suffices 
Conditions (1), (2), (3), and (*), but it generates two quartets on each of $\{a_1,a_2,b_1,b_2\}$ 
and $\{a_1,a_2,c_1,c_2\}$. It can be shown by checking the remaining 5-sets in Case 2 of our 
proof of Lemma~\ref{lem:quartet2} that every ternary map on 6 taxa satisfying Conditions 
(1), (2), (3), (*) that does not yield a thin quartet system is isomorphic to this example.  This 
raises the question whether ternary maps satisfying these four conditions can be completely 
characterized. The hope is to observe something similar to the Clebsch trees that were observed 
by Jan Weyer-Menkhoff~\cite{Weyer2003}. As a result, a phylogenetic tree with all interior vertices 
of degree 3, 5, or 6  can be reconstructed from every transitive complete quartet set.

%

\begin{figure}[ht!]
\centering  
\includegraphics[width=0.9\textwidth]{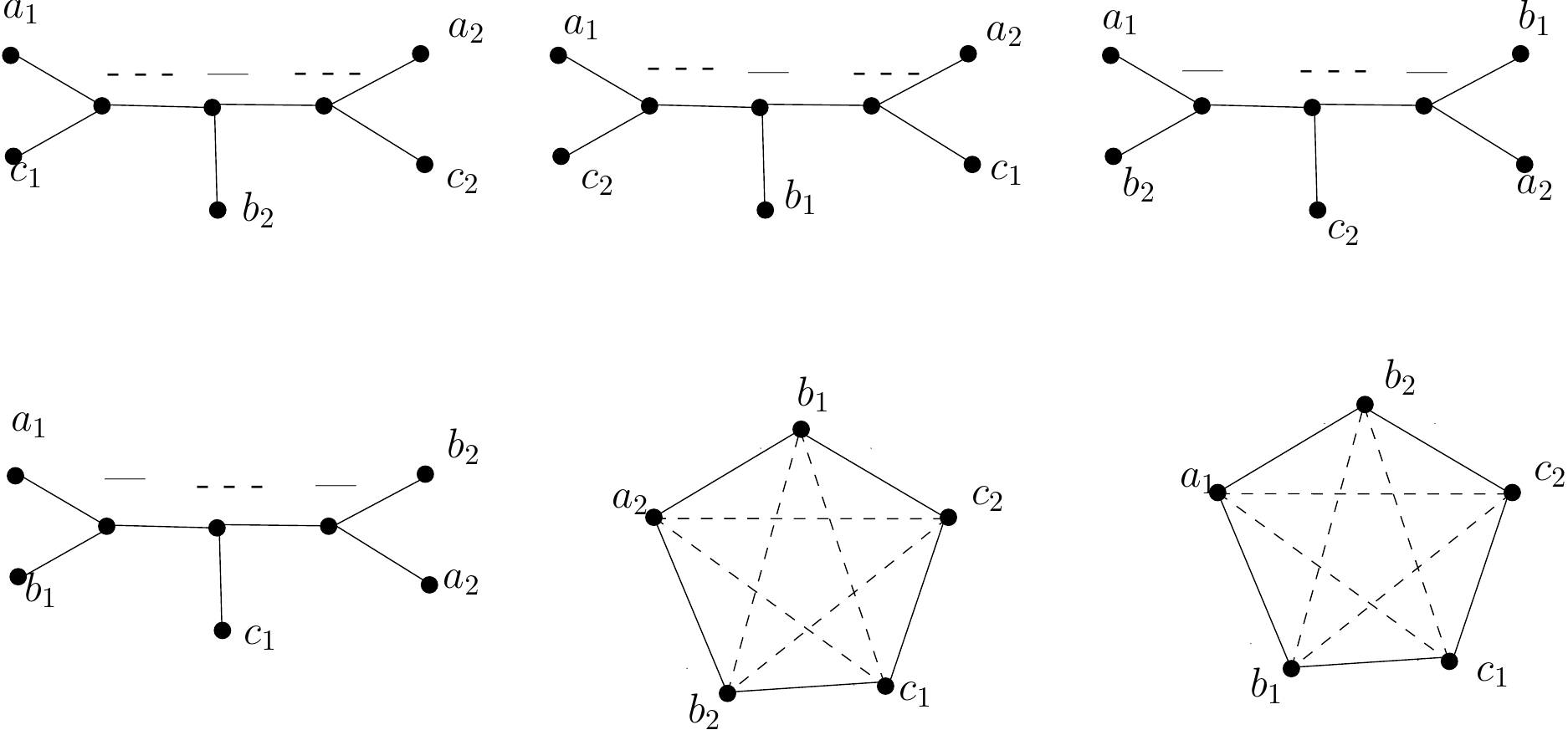} 
\caption{The 5-taxa trees respectively graph representations generated by a ternary map satisfying Conditions (1), (2), (3), (*) but not (4).} 
 \label{fig:new} 
\end{figure}

Another direction to follow up this work would be to consider more general graphs than  
trees. A {\em median graph} is a graph for which every three vertices have a unique median. 
Given a vertex-colored median graph and a subset $X$ of its vertex set, we can get 
a symmetric ternary map on $X \times X \times X$ by associating the color of the median 
to every 3-subset of $X$. It would be interesting to see whether this map can be used 
to reconstruct the underlying graph for other classes of median graphs than 
phylogenetic trees. In phylogenetics, median graphs are used to represent non-treelike 
data. Since the interior vertices of those so-called  {\em splits graphs} 
do in general not correspond to 
any ancestor of some of the taxa, reconstructing a collection of splits from the ternary 
map induced by a vertex-colored splits graph is probably limited to split systems 
that are almost compatible with a tree. 

It is one of the main observations of ~\cite{Hellmuth2013} that the 4-point condition 
for symbolic ultrametrics can be formulated in terms of {\em cographs} which 
are graphs that do 
not contain an induced path of length 3. For the special case that 
$\delta: X\times X\to M^{\odot}$ with $|M|=2$ and $m \in M$, consider the graph with 
vertex set $X$ where two vertices $x,y$ are adjacent, if and only if $\delta(x,y)=m$. 
Then deciding whether $\delta$ is a symbolic 
ultrametric can be reduced to checking whether this graph (as well as its complement) 
is a cograph. This is useful for analyzing real data which will usually not provide a perfect 
symbolic ultrametric, thus some approximation is required. For ternary maps and 
unrooted trees, the 5-taxa case looks promising, as Condition (4) translates to 
a forbidden graph representation that splits the edges of a $K_5$ into two 
5-cycles, thus we have a self-complementary forbidden induced subgraph. 
However, for more taxa, the 3-sets that are mapped to the same value of a ternary map 
$\delta$ define a 3-uniform hypergraph on $X$ and formulating Condition (4) in 
terms of this hypergraph does not seem to be promising. In addition, even if there 
are only two values of $\delta$ for 3-sets, Condition (3) does not become obsolete. 
We leave it as an open question, whether an alternative characterization of 
symbolic ternary metrics exists that makes it easier to solve the corresponding
approximation problem. 




\section{Note added in proof}

It was brought to our attention during the refereeing process that, in the context of game theory
and using different notation, Vladimir Gurvich published a result equivalent to Theorem \ref{thm:ultra-ternary}.
Already in 1984, the work was published in Russian \cite{Gurvich1984}, as well as an English
translation~\cite{Gurvich1984b}. More recently, Gurvich published another article on the topic \cite{Gurvich2009},
providing some details of the proofs that were previously omitted, and the result is included in a survey on graph
entropy by Simonyi \cite{Simonyi2001}. We point out that Gurvich’s result does not only imply our theorem, but also
Theorem \ref{thm:rooted} by B\"ocker and Dress \cite{Bocker1998} and the interpretation of symbolic ultrametrics in terms of cographs by
Hellmuth et al \cite{Hellmuth2013}.

In addition, Huber et al. \cite{Huber2017} independently published a preprint that contains the characterization of symbolic ternary
metrics. Their work, as well as Gurvich’s papers, reduces the problem to the rooted equivalent and then applies Theorem \ref{thm:rooted}.
Therefore, our quartet-based proof is the only one that stays within an unrooted setting, and the constraints on the quartet systems that
it provides may be useful for future applications.
\section*{Acknowledgements}

We thank the anonymous referees for their helpful comments and suggestions.
We thank Peter F. Stadler for suggesting to consider general median graphs
and Zeying Xu for some useful comments.
This work is supported by the NSFC (11671258) and STCSM (17690740800).
YL acknowledges support of Postdoctoral Science Foundation of China (No. 2016M601576).

\bibliographystyle{plain}
\bibliography{triple}

\end{document}